\documentclass[a4paper,12pt,reqno]{amsart}
\usepackage{amsmath,amssymb}
\usepackage{ifthen}
\usepackage{graphicx}
\usepackage{mathrsfs}
\usepackage{color}
\nonstopmode
\numberwithin{equation}{section}
\usepackage{pgfplots}
\pgfplotsset{compat=1.16}

\newtheorem{thm}{Theorem}[section]
\newtheorem{cor}[thm]{Corollary}
\newtheorem{lem}[thm]{Lemma}

\newtheorem{rem}{Remark}

\theoremstyle{definition}

\newenvironment{pf}[1][]{%
 \vskip 3mm
 \noindent
 \ifthenelse{\equal{#1}{}}%
  {{\slshape Proof. }}%
  {{\slshape #1.} }%
 }%
{\qed\bigskip}

\DeclareMathOperator{\esssup}{ess\,sup}

\makeatletter
\@namedef{subjclassname@2020}{\textup{2020} Mathematics Subject Classification}
\makeatother

\newcounter{alphabet}

\newenvironment{Thm}[1][]{\refstepcounter{alphabet}%
\bigskip%
\noindent%
{\bf Theorem \Alph{alphabet}}%
\ifthenelse{\equal{#1}{}}{}{ (#1)}%
{\bf .} \itshape}{\vskip 8pt}

\newcommand{\C}{{\mathbb C}}
\newcommand{\D}{{\mathbb D}}

\newcommand{\T}{{\mathbb T}}

\newcommand{\DD}{{\Delta}}

\renewcommand{\Re}{{\,\operatorname{Re}\,}}

\newcommand{\Gauss}{{\null_2F_1}}

\newcommand{\aand}{{\quad\text{and}\quad}}
\newcommand{\CCT}{{\mathcal{C}^1(\T)}}
\newcommand{\CT}{{\mathcal{C}(\T)}}
\newcommand{\CCD}{{\mathcal{C}^2(\D)}}
\newcommand{\Pa}{{\mathcal{P}_{\alpha}}}
\newcommand{\dth}{{\,d\theta}}

\newcounter{minutes}
\divide\time by 60
\newcounter{hours}
\multiply\time by 60 \addtocounter{minutes}{-\time}

\begin{document}
\bibliographystyle{amsplain}
\title[$\alpha$-harmonic equation]
{Schwarz Lemma and Schwarz-Pick Lemma for
solutions of the $\alpha$-harmonic equation}
\def\thefootnote{}
\footnotetext{
\texttt{\tiny File:~\jobname .tex,
          printed: \number\year-\number\month-\number\day,
          \thehours.\ifnum\theminutes<10{0}\fi\theminutes}
}
\makeatletter\def\thefootnote{\@arabic\c@footnote}\makeatother

\author[M. Li]{Ming Li}
\address{
School of Mathematics and Statistics\\
Changsha University of Science and Technology\\
960, 2nd Section, Wanjiali RD (S)
Changsha,
410004 Hunan, China}
\email{minglimath@163.com}
\author[X.-S. Ma]{Xiu-Shuang Ma}
\address{School of Mathematical Sciences \\ 
Yuehai Campus, Shenzhen University\\
3688 Nanhai Avenue, Nanshan District, Shenzhen, China}
\email{maxiushuang@szu.edu.cn \\
maxiushuang@gmail.com}
\author[L.-M.~Wang]{Li-Mei Wang}
\address{School of Statistics,
University of International Business and Economics, No.~10, Huixin
Dongjie, Chaoyang District, Beijing 100029, China}
\email{wangmabel@163.com}

\begin{abstract}
In this paper, the Schwarz type and Schwarz-Pick type
inequalities for solutions of $\alpha$-harmonic equation for $\alpha>-1$
are investigated.
By making use of the integral of trigonometric functions,
we obtain the two types of inequalities in terms of hypergeometric functions
which improve the corresponding results due to
Khalfallah et al. (Complex Var. Elliptic Equ., 2023)
and Li et al. (Bull. Malays. Math. Sci. Soc., 2022).
\end{abstract}
\thanks{The first author was supported by
National Natural Science Foundation of China (No.12001063 and No.12171055).
The third author was supported by a grant of University of International Business and Economics (No. 78210418).
}
\maketitle

\section{Introduction}

Let $\C$ denote the complex plane and $\D,~\T$ be the unit disk and the unit circle respectively.
Olofsson and Wittsten \cite{OW13} introduced the weighted Laplace operator $\DD_\alpha$ by
$$
\DD_\alpha=\partial_z(1-|z|^2)^{-\alpha} \overline{\partial}_z, \quad z\in\D
$$
for $\alpha>-1$, where
$$
\partial_z=\frac{\partial}{\partial z}=\frac1{2}\left(\frac{\partial}{\partial x}-i\frac{\partial}{\partial y}\right)\aand\overline{\partial}_z=\frac{\partial}{\partial\bar z}=\frac1{2}\left(\frac{\partial}{\partial x}+i\frac{\partial}{\partial y}\right).
$$
We mention that the formal adjoint of $\DD_\alpha$
with respect to the standard real pairing is
the differential operator
$$
\DD^*_\alpha=\overline{\partial}_z(1-|z|^2)^{-\alpha} \partial_z, \quad z\in\D
$$
which leads to the discussion similar to $\DD_\alpha$.
Note that the case $\DD_0=\partial_z\overline{\partial}_z=\overline{\partial}_z\partial_z$ is the classical Laplacian for harmonic functions. See \cite{duren:harm} for basic knowledge of harmonic functions.

Recall that the Dirichlet boundary value problem is to find a function which solves a special partial differential equation in the given domain and takes prescribed values on the boundary of the domain. For $f\in\CCD$,
of particular interest is the homogeneous equation
\begin{equation}\label{eq:alf}
\DD_\alpha f=0,\quad z\in\D,
\end{equation}
and its associated Dirichlet problem
\begin{equation}\label{eq:drcl}
\left\{
\begin{aligned}
&\DD_\alpha f=0,\quad &z\in\D, \\
&f=f^*,\quad &z\in\T.
\end{aligned}\right.
\end{equation}
Here, the boundary data $f^*\in\mathcal{D^{\prime}}(\T)$ is a distribution on $\T$, and
the boundary condition in \eqref{eq:drcl} is interpreted
in the distributional sense that  $f(re^{i\theta})\to f^*(e^{i\theta})$ in $\mathcal{D^{\prime}}(\T)$ as $r\to 1^-$. We mention here if $f$ solves the equation \eqref{eq:drcl} for $\alpha\le-1$, $f$ reduces to $0$ (see \cite[Thm. 2.3]{olofsson14}).
The solutions to \eqref{eq:alf} are called {\it $\alpha$-harmonic functions}.
Note that 0-harmonic functions are simply harmonic functions.
For more details about $\alpha$-harmonic functions,  we refer to \cite{olofsson14, OW13} .

Let $\alpha>-1$. Olofsson and Wittsten \cite{OW13} showed that
if an $\alpha$-harmonic function $f$ satisfies
$$
\lim_{r\to1^-}f(re^{i\theta})=f^*(e^{i\theta})\in\mathcal{D^{\prime}}(\T),
$$
then it has the Poisson type integral
\begin{equation}\label{eq:fpa}
f(z)=\Pa[f^*](z)
=P_{\alpha}(z)*f^*(e^{i\theta})
=\frac1{2\pi}\int_0^{2\pi} P_\alpha(ze^{-i\theta})f^*(e^{i\theta})\dth,
\end{equation}
where
\begin{equation}\label{eq:knl}
P_\alpha(z)=\frac{(1-|z|^2)^{\alpha+1}}{(1-z)(1-\bar{z})^{\alpha+1}}
\end{equation}
is the complex Poisson kernel of $\alpha$-harmonic functions and $*$ denotes the convolution of distributions on $\T$.
Note that $P_0(z)=\frac{1-|z|^2}{|1-z|^2}$ is the classical Poisson kernel.
See \cite{CV15, olofsson14} for discussions on the corresponding real kernel
$$
K_\alpha(z)=c_\alpha\frac{(1-|z|^2)^{\alpha+1}}{|1-z|^{\alpha+2}},\quad\alpha>-1,
$$
of $\alpha$-harmonic functions, where
\begin{equation}\label{eq:c}
c_\alpha=\frac{\Gamma\left(\frac{\alpha}{2}+1\right)^2}{\Gamma(\alpha+1)}
\end{equation}
and $\Gamma(x)=\int_0^{\infty}t^{x-1}e^{-x}\,dt$ $(x>0)$ is the Gamma function.
The integral representation (\ref{eq:fpa}) of $\alpha$-harmonic functions plays a vital role
in the present paper.

The Schwarz lemma for analytic functions is extremely important in complex analysis.
Heinz \cite{heinz59} generalized the result to harmonic functions,
i.e., if $f$ is harmonic in $\D$ with $f(0)=0$ and $|f(z)|<1$, then
\begin{equation}\label{eq:4pi}
|f(z)|\le\frac4{\pi}\arctan{|z|},\quad z\in\D.
\end{equation}
Hethcode \cite{Hethcote77} 
dropped the condition $f(0)=0$ and proved that
\begin{equation}\label{eq:4pi0}
\left|f(z)-\frac{1-|z|^2}{1+|z|^2}f(0)\right|\le\frac4{\pi}\arctan{|z|},\quad z\in\D,
\end{equation}
for harmonic function $f$ in $\D$ satisfying $|f(z)|<1$.
Li, Rasila and Wang \cite{LRW20} struggled to generalize the Schwarz lemma to $\alpha$-harmonic functions. Li and Chen \cite{LC22} found a gap in \cite{LRW20},
and proved the following result.

\begin{Thm}$($\cite[Thm.1.1]{LC22}$)$\label{thm:lc11}
Suppose that $f^*\in\CCT$. If $f\in\CCD$ satisfies \eqref{eq:drcl} with $\alpha>-1$, then
 $|f(z)| \le M_1(|z|,\alpha)||f^*||_{\infty}$ for $z\in\D$, where
 $||f^*||_{\infty}=\mathop{\esssup}_{z\in\T}|f^*(z)|$,
\begin{equation*}\label{eq:lc11}
M_1(|z|,\alpha)=\left\{
\begin{aligned}
&\frac{2^{1+\alpha}}{\pi}\arctan\left(\frac{1+|z|}{1-|z|}\tan\frac{c\pi}2\right), &\alpha\ge0, \\
&\frac{2^{1-\alpha}}{\pi}(1-|z|^2)^{\alpha}\arctan\left(\frac{1+|z|}{1-|z|}\tan\frac{c\pi}2\right), &\quad -1<\alpha<0,
\end{aligned}\right.
\end{equation*}
and $c=\mathcal{P}_\alpha [|f^*|](0)/||f^*||_{\infty}$.
\end{Thm}

The upper bound $M_1(|z|,\alpha)$ in Theorem A is not very explicit and the number $c$ depends on the individual function $f$.
By observing that $c=0$ if and only if $f^{*}(e^{i\theta})\equiv 0$, we find that 
 $M_1(|z|,\alpha)\to +\infty$ for $-1<\alpha<0$ as $|z|\to 1^-$ unless $f^{*}(e^{i\theta})= 0$.
Recently, Khalfallah and Mateljević \cite{KM2023} constructed
a Schwarz type inequality involving $f(0)$ instead of $\mathcal{P}_\alpha [|f^*|](0)$
as follows.

\begin{Thm}$($\cite[Thm.3.2]{KM2023}$)$
Let $\alpha>-1$ and $f: \D \rightarrow \D$ be an $\alpha$-harmonic function. Then
$$
\left|f(z)-\frac{\left(1-|z|^2\right)^{\alpha+1}}{1+|z|^2} f(0)\right|\le M_2(|z|,\alpha)
$$
with
\begin{eqnarray*}
 M_2(|z|,\alpha)
=
\begin{cases}
&\dfrac{2^{\alpha+2}}{\pi} \arctan |z|+2^{\alpha+1}(1-|z|)\left(1-(1-|z|)^\alpha\right),
\quad  \alpha \geq 0,\\
& \dfrac{4}{\pi}(1-|z|)^\alpha \arctan |z|+(1-|z|)^\alpha-1,\qquad \qquad \quad
\alpha<0.
\end{cases}
\end{eqnarray*}
\end{Thm}

Colonna \cite{Colonna:89} generalized the classical Schwarz-Pick lemma for holomorphic functions to harmonic functions of the unit disk as follows.

\begin{Thm}$($\cite[Thm.3]{Colonna:89}$)$
If $f$ is a harmonic function from $\D$ into $\D$, then
 for $z\in\D$
\begin{equation}\label{eq:Colo}
||D_f(z)||\leq \frac{4}{\pi}\frac{1}{1-|z|^2},
\end{equation}
where $||D_f(z)||=|f_z(z)|+|f_{\bar z}(z)|$.
\end{Thm}

The  Schwarz-Pick type inequality for $\alpha$-harmonic
functions was studied by Li et al. \cite{LWX17}, and generalized
in \cite{LC22} and \cite{LRW20}.
Here we only mention the latest result on this topic.

\begin{Thm}$($\cite[Thm.1.2]{LC22}$)$\label{Thm:S-P}
Suppose that $f^*\in\CCT$. If $f\in\CCD$ satisfies \eqref{eq:drcl} with $\alpha>-1$, then for $z\in\D$,
\begin{equation*}\label{eq:lc11}
||D_f(z)||
\le\left\{
\begin{aligned}
&\frac{(1+\alpha)2^{1+\alpha}}{1-|z|^2}||f^*||_{\infty}, &\alpha\ge0, \\
&\frac{2^{1-\alpha}}{(1-|z|^2)^{1-\alpha}}||f^*||_{\infty},\quad  &\quad -1<\alpha<0.
\end{aligned}\right.
\end{equation*}
\end{Thm}

For the related topics on Schwarz lemma and Schwarz-Pick lemma, see \cite{Chen2011, KC2012}.

Our first aim is to establish a Schwarz type inequality (see Theorem \ref{thm:swzi} below)
in the spirit of \cite{KM2023} and give a uniform bound for $\alpha$-harmonic functions in Theorem D.
Secondly, refined Schwarz-Pick
type inequalities for $\alpha$-harmonic functions
are provided in Theorem \ref{thm:swz-pk} and Corollary \ref{cor:swz-pk}.

%
%

Before stating the main results, we introduce the definition of hypergeometric functions here.
For complex numbers $a,~b$ and $c$ with $c\neq 0,-1,-2,\dots$, the hypergeometric function is defined by the power series
$$
\Gauss(a,b;c;x)=\sum_{n=0}^{\infty}\frac{(a)_n(b)_n}{(c)_n n!}x^n,\quad |x|<1,
$$
where $(a)_0=1$, $(a)_{n}=a(a+1)\dots(a+n-1)$ for $n=1,2,\dots$ are the Pochhammer symbols.
Note that $\Gauss(a,b;c;x)=\Gauss(b,a;c;x)$ by definition.
An important result on hypergeometric functions is Gauss summation formula
\begin{equation}\label{eq:gauss-prop2}
\lim_{x\to1^-} \Gauss(a,b;c;x)=\frac{\Gamma(c)\Gamma(c-a-b)}{\Gamma(c-a)\Gamma(c-b)}=\frac{B(c,c-a-b)}{B(c-a,c-b)}
\end{equation}
if $\Re{(a+b-c)}<0$ (see \cite[p.556, 15.1.20]{AbramowitzStegun:1965}), where $B(x,y)$ denotes the beta function. 
One can find more on special functions in \cite{AbramowitzStegun:1965, AAR99:spcl}.


\section{Statement of main results}
The main results are stated in this section.
The following two theorems provide the Schwarz type inequalities for $\alpha$-harmonic functions.

\begin{thm}\label{thm:swzi}
Suppose that $f^*\in\CT$. If $f\in\CCD$ satisfies \eqref{eq:drcl} with $\alpha>-1$ and maps $\D$ into $\D$, then for $z\in\D$,
$$
\left|f(z)-\frac{(1-r^2)^{\alpha+1}}{1+r^2}f(0)\right|\\
\leq M(r,\alpha),
$$
where $r=|z|$ and $M(r,\alpha)$ is
\begin{equation*}
\left\{
\begin{aligned}
&\dfrac{(1-r^2)^{\alpha+1}|(1-r)^{-\alpha}-1|}{1+r^2}+\dfrac{2^{2+\frac{\alpha}{2}} r(1+r^2)^{\frac{\alpha}{2}-1}}{\pi}
\Gauss\left(\frac{1}{2},\frac{1}{2}-\frac{\alpha}{2};\frac{3}{2};\frac{4r^2}{(1+r^2)^2}\right),
\quad \alpha\ge 0,\\
&\dfrac{(1-r^2)^{\alpha+1}|(1-r)^{-\alpha}-1|}{1+r^2}
+\dfrac{4 r}{\pi}(1+r^2)^{\frac{\alpha}{2}-1}
\Gauss\left(\frac{1}{2},\frac{1}{2}-\frac{\alpha}{2};\frac{3}{2};\frac{4r^2}{(1+r^2)^2}\right),
\quad -1<\alpha<0.
\end{aligned}\right.
\end{equation*}
\end{thm}

In view of the equation \eqref{eq:gauss-prop2},
we find that $\displaystyle \lim_{r\to1^-}M(r,\alpha)$
is bounded for $\alpha>-1$.
Figure 1 is produced by Mathematica
demonstrating the upper bound $M_2(r,\alpha)$ in Theorem B
and $M(r,\alpha)$ in Theorem \ref{thm:swzi}.
Thus Theorem \ref{thm:swzi} refines the Schwarz
type inequality in Theorem B.

\begin{figure}[htbp]
\begin{center}
\includegraphics[width=.6\textwidth]{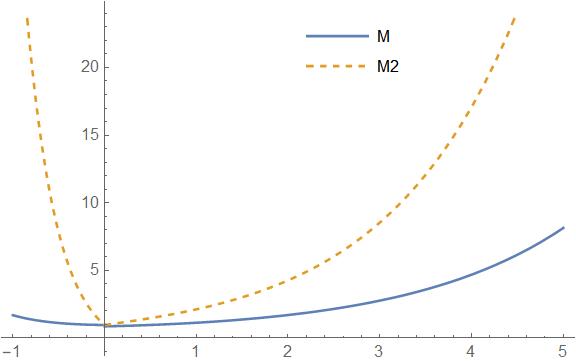}
\end{center}
\caption{The upper bounds in Theorems B and \ref{thm:swzi} for $r=0.99$.}
\label{Figure 1}
\end{figure}

\begin{rem}
Putting $\alpha=0$ in Theorem \ref{thm:swzi}, the upper bound reduces to
\begin{align*}
\frac{4r}{(1+r^2)\pi}\frac{\arcsin \frac{2r}{1+r^2}}{\frac{2r}{1+r^2}}
=\frac{2}{\pi}\arcsin \frac{2r}{1+r^2}
=\frac{4}{\pi}\arctan r,
\end{align*}
which is exactly the Hethcode's Schwarz type inequality $($\ref{eq:4pi0}$)$ for harmonic functions.
\end{rem}

We next obtain the following concise form for estimate in Theorem \ref{thm:swzi},
which provides a better upper bound than the one in
\cite[Cor.3.1]{KM2023}.
\begin{cor}\label{cor:swzi}
Suppose that $f^*\in\CT$. If $f\in\CCD$ satisfies \eqref{eq:drcl} with $\alpha\geq 0$ and maps $\D$ into $\D$, then
$$
\left|f(z)-\frac{(1-|z|^2)^{\alpha+1}}{1+|z|^2}f(0)\right|
\leq M'(|z|,\alpha)
$$
holds for $z\in\D$ with
\begin{eqnarray*}
M'(|z|,\alpha)
=
\begin{cases}
&2^{1+\alpha}|z|\left(\dfrac{1}{\pi}+\alpha\right),
\quad\qquad   \quad \alpha \geq 2,\\
&|z|\left(\dfrac{2^{2+\frac{\alpha}{2}}}{\pi}+2^{1+\alpha}\alpha\right),
\quad \quad\quad  1\le \alpha<2,\\
& 2^{1+\frac{\alpha}{2}}|z|+2^{1+\alpha}(1-|z|)|z|^\alpha ,
\quad  0\le \alpha <1.
\end{cases}
\end{eqnarray*}
\end{cor}

\begin{thm}\label{thm:swz}
Suppose that $f^*\in\CT$. If $f\in\CCD$ satisfies \eqref{eq:drcl} with $\alpha>-1$, then for $z\in\D$,
\begin{equation}\label{eq:swz}
|f(z)| 
\le\Gauss\left(-\frac{\alpha}{2},-\frac{\alpha}{2};1;|z|^{2}\right)||f^*||_{\infty}.
\end{equation}
\end{thm}

It follows from the summation formula (\ref{eq:gauss-prop2}) that
\begin{equation}\label{eq:sum}
\displaystyle\lim_{r\to1^-}
\Gauss\left(-\frac{\alpha}{2},-\frac{\alpha}{2};1;r\right)=
\frac{1}{c_{\alpha}}, \quad \alpha>-1
\end{equation}
where $c_{\alpha}$ is given in \eqref{eq:c}.
We then infer from the identity \eqref{eq:sum} and Theorem \ref{thm:swz} the following $L^1$-means bound of the $\alpha$-harmonic Poisson kernel $P_{\alpha}$, which was also proved by Olofsson and Wittsten \cite[Cor. 2.8]{OW13} (see also \cite[Thm. 3.2]{olofsson14}
for the corresponding result of $K_{\alpha}$).
But our proof is more straightforward.

\begin{Thm}$($\cite[Cor.2.8]{OW13}$)$\label{Thm:unibound}
Let $\alpha>-1$. For $0 \leq r<1$ we have
$$
\begin{aligned}
\frac{1}{2 \pi} \int_{\mathbb{T}}\left|P_\alpha\left(r e^{i \theta}\right)\right|  d \theta & \leq \frac{1}{c_{\alpha}}\quad \text { and } \quad
\lim _{r \rightarrow 1^-} \frac{1}{2 \pi} \int_{\mathbb{T}}\left|P_\alpha\left(r e^{i \theta}\right)\right| d \theta
& =\frac{1}{c_{\alpha}},
\end{aligned}
$$
where $c_{\alpha}$ is given in \eqref{eq:c}.
\end{Thm}

We turn to establish the Schwarz-Pick type inequality for $\alpha$-harmonic functions by making use of integrals of trigonometric functions.

\begin{thm}\label{thm:swz-pk}
Let $f$ be an $\alpha$-harmonic function in $\D$ with $\alpha>-1$ satisfying \eqref{eq:drcl} and  $f^*\in\CT$.
Then for $z\in\D$,
\begin{equation*}\label{eq:esti1}
||D_f(z)||
\le\left\{
\begin{aligned}
&\frac{2(1+\alpha)||f^*||_\infty}{1-|z|^2}\Gauss\left(-\frac{\alpha}{2},-\frac{\alpha}2;1;|z|^2\right), &\alpha\ge0, \\
&\frac{2||f^*||_\infty}{1-|z|^2}\Gauss\left(-\frac{\alpha}{2},-\frac{\alpha}2;1;|z|^2\right), &\quad -1<\alpha<0.
\end{aligned}\right.
\end{equation*}
\end{thm}

As an easy consequence of Theorem \ref{thm:swz-pk}
and the identity \eqref{eq:sum}, we obtain the following result.
\begin{cor}\label{cor:swz-pk}
Let $f$ be an $\alpha$-harmonic function in $\D$ with $\alpha>-1$ satisfying \eqref{eq:drcl} and $f^*\in\CT$.
Then for $z\in\D$,
\begin{equation}\label{eq:bound}
||D_f(z)||
\leq \left\{
\begin{aligned}
&\frac{2(1+\alpha)}{c_{\alpha}}\frac{||f^*||_\infty}{1-|z|^2}, \quad &\alpha\geq 0,\\
&\frac{2}{c_{\alpha}}\frac{||f^*||_\infty}{1-|z|^2}, \quad & -1<\alpha<0,
\end{aligned}\right.
\end{equation}
where $c_{\alpha}$ is given in \eqref{eq:c}.
\end{cor}

\begin{rem}
Olofsson and Wittsten \cite[Pro.2.9]{OW13} proved that
\begin{equation*}\label{eq:c1}
\frac{1}{c_{\alpha}}
=\frac{2^{\alpha}\Gamma\left(\frac{1}{2}
+\frac{\alpha}{2}\right)}{\sqrt{\pi}\Gamma\left(1+\frac{\alpha}{2}\right)},
\end{equation*}
which is obviously less than
$2^{\alpha}$ for $\alpha\geq 0 $.
Therefore  Corollary \ref{cor:swz-pk} refines Theorem D.
\end{rem}


\section{Proof of Shcwarz type inequalities}

In this section, we show the proofs of Theorems
\ref{thm:swzi}, \ref{thm:swz} and their corollaries.
In order to do that, the following auxiliary lemmas are required.

\begin{lem}[$\text{\cite[Lem.2.1]{olofsson14}}$]\label{lem:mono}
Let $c>0, ~a \leq c,~ b \leq c$ and $a b \leq 0~(a b \geq 0)$.
Then the function $\Gauss(a, b ; c ; x)$ is decreasing $($increasing$)$ on $x \in(0,1)$.
\end{lem}

\begin{lem}[$\text{\cite[Thm.2.31]{PonVuor:1997}}$]\label{lem:mono2}
For each $A, ~B,~ C, ~a,~ b, ~c \in(0, +\infty)$, let
$$
g(x)=\frac{\Gauss(A,B;C;x)}{\Gauss(a, b;c; x)}.
$$
If $C+a+b \geqslant c+A+B$ and $(a-1)(b-1)(1-C)<(A-1)(B-1)(1-c)$,
the function $g(x)$ is decreasing for $x \in(0,1)$
if and only if $C a b+(a-1)(b-1)(1-C)\geqslant c A B+(A-1)(B-1)(1-c)$.
\end{lem}

Let $\binom{\alpha}{n}$ denote the generalized binomial coefficient,
i.e.,
$$
\binom{\alpha}{n}=\frac{\Gamma(\alpha+1)}{\Gamma(\alpha-n+1)n!}
=\frac{(\alpha-n+1)_n}{n!}.
$$

\begin{lem}\label{lem:exp}
For real numbers $a,~b,~\alpha$ and $\beta$ satisfying
$a,~b\in(-1,1)$ and $\alpha,~\beta\geq 0$, the equality
\begin{equation}\label{eq:exp}
\frac{1}{2\pi}\int_0^{2\pi} \frac{(1-a\cos\theta)^\alpha}{(1-b\cos\theta)^{\beta}}d\theta
=\sum_{n=0}^{\infty}\sum_{m=0}^{2n}
\binom{\alpha}{m}\binom{-\beta}{2n-m}
\frac{(2n)!}{4^n(n!)^2}a^mb^{2n-m}
\end{equation}
holds, where we take the convention that $0^0=0!=1$.
In particular, the equality
\begin{equation}\label{eq:exp2}
\frac{1}{2\pi}\int_0^{2\pi} \frac{d\theta}{|1-ze^{i\theta}|^{2\beta}}
=(1-|z|^2)^{1-2\beta}
\Gauss\left(1-\beta,1-\beta;1;|z|^2\right), \quad z\in\D
\end{equation}
is valid.
\end{lem}

\begin{proof}
By using the power series expansion
$$
(1+x)^{\alpha}=\sum_{n=0}^{\infty}\binom{\alpha}{n}x^n, \quad -1<x<1,
$$
we have
$$
\frac{(1-ax)^\alpha}{(1-bx)^{\beta}}=
\sum_{n=0}^{\infty}\sum_{m=0}^{n}
\binom{\alpha}{m}\binom{-\beta}{n-m}a^mb^{n-m}(-x)^n, \quad -1<x<1.
$$
In view of the above expression and the formula
\begin{equation}\label{eq:cos}
\int_0^{\frac{\pi}{2}}\cos^n\theta d\theta=
\begin{cases}
\dfrac{(2k-1)!!}{(2k)!!}\dfrac{\pi}{2},\quad n=2k,\\
\dfrac{(2k)!!}{(2k+1)!!}, \quad \quad n=2k+1
\end{cases}
\end{equation}
for $k=0,1,2\ldots$ with $(n)!!$ being the double factorial of $n$,
 we obtain
\begin{eqnarray*}
\frac{1}{2\pi}\int_0^{2\pi} \frac{(1-a\cos\theta)^\alpha}{(1-b\cos\theta)^{\beta}}d\theta
&=&\frac{1}{2\pi}\sum_{n=0}^{\infty}\sum_{m=0}^{n}
\binom{\alpha}{m}\binom{-\beta}{n-m}a^mb^{n-m}(-1)^n \int_0^{2\pi}\cos^n\theta d\theta\\
&=&\sum_{n=0}^{\infty}\sum_{m=0}^{2n}
\binom{\alpha}{m}\binom{-\beta}{2n-m}
\frac{(2n)!}{4^n(n!)^2}a^mb^{2n-m},
\end{eqnarray*}
as claimed in \eqref{eq:exp}.

We turn to prove the equation \eqref{eq:exp2}.
It suffices to check it for $z=r\in[0,1)$.
Letting $a=0$ and $b=2r/(1+r^2)$ in \eqref{eq:exp}, we have
\begin{eqnarray*}
\frac{1}{2\pi}\int_0^{2\pi} \frac{d\theta}{|1-re^{i\theta}|^{2\beta}}
&=&\frac{1}{2\pi(1+r^2)^{\beta}}\int_0^{2\pi} \frac{d\theta}{(1-b\cos\theta)^{\beta}}\\
&=&\frac{1}{(1+r^2)^{\beta}}\sum_{n=0}^{\infty}\binom{-\beta}{2n}
\frac{(2n)!}{4^n(n!)^2}b^{2n}\\
&=&\frac{1}{(1+r^2)^{\beta}}
\Gauss\left(\frac{\beta}{2},\frac{1}{2}+\frac{\beta}{2};1;b^2\right).
\end{eqnarray*}
We then infer from the Euler transform
(see \cite[p.559, 15.3.3]{AbramowitzStegun:1965})
\begin{equation}\label{eq:Euler}
\Gauss(a,b;c;x)=(1-x)^{c-a-b}\Gauss(c-a,c-b;c;x), \quad |x|<1
\end{equation}
and the quadratic transformation formula
(see \cite[p.560, (15.3.19)]{AAR99:spcl})
\begin{equation}\label{eq:quadratic}
\Gauss\left(a,a+\frac{1}{2};c;x\right)
=\left(\frac{1}{2}+\frac{\sqrt{1-x}}{2}\right)^{-2a}
\Gauss\left(2a,2a-c+1;c;\frac{1-\sqrt{1-x}}{1+\sqrt{1-x}}\right),
\quad |x|<1
\end{equation}
that
\begin{eqnarray*}
\frac{1}{2\pi}\int_0^{2\pi} \frac{d\theta}{|1-re^{i\theta}|^{2\beta}}
&=&\frac{(1-r^2)^{1-2\beta}}{(1+r^2)^{1-\beta}}
\Gauss\left(1-\frac{\beta}{2},\frac{1}{2}-\frac{\beta}{2};1;b^2\right)\\
&=&(1-r^2)^{1-2\beta}
\Gauss\left(1-\beta,1-\beta;1;r^2\right).
\end{eqnarray*}
Hence the proof is complete.
\end{proof}

Note that the identity \eqref{eq:exp2} was
proved by Olofsson and Wittsten in \cite[Lem.2.3]{OW13}
using Parseval's formula (see also (2.6) in \cite{KM2022}),
although the expressions are a little different.

Now we are ready to show the proofs.

\begin{pf}[Proof of Theorem \ref{thm:swzi}]
It follows from the integral form (\ref{eq:fpa}) that the class of $\alpha$-harmonic functions is rotationally invariant, i.e., if $f(z)$ is  $\alpha$-harmonic, so is $f(e^{i\theta}z)$
for $\forall~ \theta\in[0,2\pi)$.
Thus we only prove the claimed inequality is valid for $z=r\in[0,1)$.

By making use of the integral form (\ref{eq:fpa}) of $\alpha$-harmonic function $f$, we have
\begin{eqnarray*}
&&\left|f(r)-\frac{(1-r^2)^{\alpha+1}}{1+r^2}f(0)\right|\\
&=&\frac1{2\pi}\left|\int_0^{2\pi} \frac{(1-r^2)^{\alpha+1}}{(1-re^{-i\theta})(1-re^{i\theta})^{\alpha+1}}f^*(e^{i\theta})\dth-\int_0^{2\pi} \frac{(1-r^2)^{\alpha+1}}{1+r^2}f^*(e^{i\theta})\dth
\right| \\
&\le&\frac1{2\pi}\frac{(1-r^2)^{\alpha+1}}{1+r^2}||f^*||_{\infty}\int_0^{2\pi}\left|\frac{(1-re^{-i\theta})(1-re^{i\theta})^{\alpha+1}-1-r^2}{(1-re^{-i\theta})(1-re^{i\theta})^{\alpha+1}}\right|\dth
\\
&\le&\frac1{2\pi}\frac{(1-r^2)^{\alpha+1}}{1+r^2}
\int_0^{2\pi}\frac{\left|(1-re^{-i\theta})(1-re^{i\theta})^{\alpha+1}-1-r^2\right|}
{|1-re^{i\theta}|^{\alpha+2}}\dth,
\end{eqnarray*}
since $f$ maps $\D$ into $\D$ by assumption.
Let $\zeta=re^{-i\theta}$ for the sake of brevity.
The integrant of the above integral can be estimated as follows
\begin{eqnarray*}
\frac{\left|(1-re^{-i\theta})(1-re^{i\theta})^{\alpha+1}-1-r^2\right|}
{|1-re^{i\theta}|^{\alpha+2}}
&=&\frac{\left||1-\zeta|^2(1-\zeta)^{\alpha}-1-r^2\right|}
{|1-\zeta|^{\alpha+2}}\\
&\le& \frac{\left|1-\zeta|^2|(1-\zeta)^{\alpha}-1\right|+|\zeta+\bar{\zeta}|}
{|1-\zeta|^{\alpha+2}}\\
&=&\left|1-(1-\zeta)^{-\alpha}\right|+\frac{|\zeta+\bar{\zeta}|}
{|1-\zeta|^{\alpha+2}}\\
&\le& |(1-r)^{-\alpha}-1|+\frac{|\zeta+\bar{\zeta}|}
{|1-\zeta|^{\alpha+2}},
\end{eqnarray*}
since all the Taylor coefficients of $(1-\zeta)^{-\alpha}-1$ are non-negative
for $\alpha\ge 0$ and negative for $-1<\alpha<0$.
In view of the identity (\ref{eq:cos})
, we find by computations that
\begin{eqnarray*}
\frac1{2\pi}\int_0^{2\pi}
\frac{|\zeta+\bar{\zeta}|}{|1-\zeta|^{\alpha+2}}\dth
&=&\frac{r}{\pi(1+r^2)^{\frac{\alpha}{2}+1}}\int_0^{2\pi}\frac{|\cos\theta|}{(1-R\cos\theta)^{\frac{\alpha}2+1}}\dth\\
&=&\frac{r}{\pi(1+r^2)^{\frac{\alpha}{2}+1}}\int_0^{2\pi}
|\cos\theta|\sum_{n=0}^{\infty}\frac{\left(-\frac{\alpha+2}{2}\right)_n}{n!}(-R\cos\theta)^{n}\dth\\
&=&\frac{r}{\pi(1+r^2)^{\frac{\alpha}{2}+1}}
\sum_{n=0}^{\infty}\frac{\left(-\frac{\alpha+2}{2}\right)_{2n}}{(2n)!}R^{2n}\int_0^{2\pi}|\cos\theta|\cos^{2n}\theta\dth\\
&=&\frac{4r}{\pi(1+r^2)^{\frac{\alpha}{2}+1}}
\sum_{n=0}^{\infty}\frac{\left(-\frac{\alpha+2}{2}\right)_{2n}}{(2n)!}\frac{(2n)!!}{(2n+1)!!}R^{2n}\\
&=&\frac{4r}{\pi(1+r^2)^{\frac{\alpha}{2}+1}}
\sum_{n=0}^{\infty}\frac{\left(\frac{1}{2}+\frac{\alpha}{4}\right)_n
\left(1+\frac{\alpha}{4}\right)_n}{\left(\frac{1}{2}\right)_n\left(\frac{3}{2}\right)_n}R^{2n},
\end{eqnarray*}
with $R=2r/(1+r^2)$.
Denote
$$
Q_n=\frac{\left(\frac{1}{2}+\frac{\alpha}{4}\right)_n\left(1+\frac{\alpha}{4}\right)_n}
{\left(\frac{1}{2}\right)_n\left(1+\frac{\alpha}{2}\right)_n},
\quad n=0,1,2\ldots
$$
for simplicity.
By the proof of Lemma 2.11 in \cite{SWW23} (see also \cite[p.~283]{PonVuor:1997}),
 we see that the sequence
$\{Q_n\}$ is increasing since $\alpha\geq 0$ and decreasing for $-1<\alpha<0$, and
$$
\lim_{n\to\infty}Q_n=\frac{B\left(\frac{1}{2},1+\frac{\alpha}{2}\right)}
{B\left(\frac{1}{2}+\frac{\alpha}{4},1+\frac{\alpha}{4}\right)}
=\frac{\sqrt{\pi}\Gamma\left(1+\frac{\alpha}{2}\right)}
{\Gamma\left(\frac{1}{2}+\frac{\alpha}{4}\right)
\Gamma\left(1+\frac{\alpha}{4}\right)},
$$
since $\Gamma(1/2)=\sqrt{\pi}$.
Applying the duplication formula
$$
\Gamma(2 x)=\frac{2^{2 x-1}}{\sqrt{\pi}} \Gamma(x) \Gamma\left(x+\frac{1}{2}\right)
$$
for the Gamma function $\Gamma\left(1+\alpha/2\right)$
$($see \cite[Section 1.5]{AAR99:spcl}$)$,
we have
\begin{equation*}
\frac{\sqrt{\pi}\Gamma\left(1+\frac{\alpha}{2}\right)}
{\Gamma\left(\frac{1}{2}+\frac{\alpha}{4}\right)
\Gamma\left(1+\frac{\alpha}{4}\right)}
=2^{\frac{\alpha}{2}}.
\end{equation*}
Thus we deduce from the above observations and the transform \eqref{eq:Euler} that
\begin{eqnarray*}
\sum_{n=0}^{\infty}\frac{\left(\frac{1}{2}+\frac{\alpha}{4}\right)_n
\left(1+\frac{\alpha}{4}\right)_n}{\left(\frac{1}{2}\right)_n\left(\frac{3}{2}\right)_n}R^{2n}
&=&\sum_{n=0}^{\infty}\left(Q_n-2^{\frac{\alpha}{2}}\right)
\frac{\left(1+\frac{\alpha}{2}\right)_n}{\left(\frac{3}{2}\right)_n}R^{2n}
+2^{\frac{\alpha}{2}}\sum_{n=0}^{\infty}
\frac{\left(1+\frac{\alpha}{2}\right)_n}{\left(\frac{3}{2}\right)_n}R^{2n}\\
&\le &\Gauss\left(1,1+\frac{\alpha}{2};\frac{3}{2};R^{2}\right)\\
&=&2^{\frac{\alpha}{2}}\left(\frac{1+r^2}{1-r^2}\right)^{1+\alpha}
\Gauss\left(\frac{1}{2},\frac{1}{2}-\frac{\alpha}{2};\frac{3}{2};R^{2}\right),
\end{eqnarray*}
for $\alpha\ge 0$, while
\begin{eqnarray*}
\sum_{n=0}^{\infty}\frac{\left(\frac{1}{2}+\frac{\alpha}{4}\right)_n
\left(1+\frac{\alpha}{4}\right)_n}{\left(\frac{1}{2}\right)_n\left(\frac{3}{2}\right)_n}R^{2n}
&=&\sum_{n=0}^{\infty}\left(Q_n-Q_0\right)
\frac{\left(1+\frac{\alpha}{2}\right)_n}{\left(\frac{3}{2}\right)_n}R^{2n}
+\sum_{n=0}^{\infty}
\frac{\left(1+\frac{\alpha}{2}\right)_n}{\left(\frac{3}{2}\right)_n}R^{2n}\\
&\le &\Gauss\left(1,1+\frac{\alpha}{2};\frac{3}{2};R^{2}\right)\\
&=&\left(\frac{1+r^2}{1-r^2}\right)^{1+\alpha}
\Gauss\left(\frac{1}{2},\frac{1}{2}-\frac{\alpha}{2};\frac{3}{2};R^{2}\right),
\end{eqnarray*}
for $-1<\alpha<0$.
Therefore we finally arrive at
\begin{eqnarray*}
\left|f(r)-\frac{(1-r^2)^{\alpha+1}}{1+r^2}f(0)\right|
\leq M(r,\alpha)
\end{eqnarray*}
as claimed.
Thus we are done.
\end{pf}

\begin{pf}[Proof of Corollary \ref{cor:swzi}]
We first verify the case $0\leq \alpha <1$.
After letting $z=r$ and changing variable $R=2r/(1+r^2)$,
we see that
$$r^2=\frac{1-\sqrt{1-R^2}}{1+\sqrt{1-R^2}}.$$
Thus the following expression in the upper bound $M(r,\alpha)$ of Theorem \ref{thm:swzi}
can be written as
\begin{eqnarray*}
(1+r^2)^{\frac{\alpha}{2}-1}
\Gauss\left(\frac{1}{2},\frac{1}{2}-\frac{\alpha}{2};\frac{3}{2};\frac{4r^2}{(1+r^2)^2}\right)
&=&\left(1+\frac{1-\sqrt{1-R^2}}{1+\sqrt{1-R^2}}\right)^{\frac{\alpha}{2}-1}
\Gauss\left(\frac{1}{2},\frac{1}{2}-\frac{\alpha}{2};\frac{3}{2};R^2\right)\\
&=&\frac{\Gauss\left(\frac{1}{2},\frac{1}{2}-\frac{\alpha}{2};\frac{3}{2};R^2\right)}
{\left(\frac{1+\sqrt{1-R^2}}{2}\right)^{\frac{\alpha}{2}-1}}\\
&=&\frac{\Gauss\left(\frac{1}{2},\frac{1}{2}-\frac{\alpha}{2};\frac{3}{2};R^2\right)}
{\Gauss\left(\frac{1}{2}-\frac{\alpha}{4},1-\frac{\alpha}{4};2-\frac{\alpha}{2};R^2\right)}.
\end{eqnarray*}
An elementary calculation shows that the above quotient function
satisfies the conditions of Lemma \ref{lem:mono2} if $0\le \alpha<1$,
and thus it is decreasing for $R\in(0,1)$.
Hence the inequality
$$
(1+r^2)^{\frac{\alpha}{2}-1}
\Gauss\left(\frac{1}{2},\frac{1}{2}-\frac{\alpha}{2};\frac{3}{2};\frac{4r^2}{(1+r^2)^2}\right)
\leq 1
$$
holds for $r\in[0,1)$.
Therefore we conclude that
\begin{eqnarray*}
M(r,\alpha)&\le&
 2^{1+\alpha}(1-r)[1-(1-r)^{\alpha}]+\frac{2^{2+\frac{\alpha}{2}}r}{\pi}\\
&\le & 2^{1+\alpha}(1-r)r^{\alpha}+\frac{2^{2+\frac{\alpha}{2}}r}{\pi}
\end{eqnarray*}
since $(1-r)^{\alpha}+r^{\alpha}\ge 1$ for $\alpha\in[0,1)$.

We proceed to show the case $\alpha\geq 1$.
We deduce from Lemma \ref{lem:mono} that the hypergeometric
function $\Gauss\left(\frac{1}{2},\frac{1}{2}-\frac{\alpha}{2};\frac{3}{2};t\right)$
is
decreasing if $\alpha\geq 1$, which implies that
$$
\Gauss\left(\frac{1}{2},\frac{1}{2}-\frac{\alpha}{2};\frac{3}{2};t\right)
\leq 1
$$
for $t\in[0,1)$.
Thus the bound $M(r,\alpha)$ in Theorem \ref{thm:swzi} can
be estimated as
\begin{eqnarray*}
M(r,\alpha)&\le&
 2^{1+\alpha}[(1-r)-(1-r)^{\alpha+1}]
 +\frac{2^{2+\frac{\alpha}{2}}r(1+r^2)^{\frac{\alpha}{2}-1}}{\pi}\\
&\le &
\begin{cases}
2^{1+\alpha}\alpha r +\dfrac{2^{2+\frac{\alpha}{2}}r}{\pi},\quad 1\le \alpha<2\\
2^{1+\alpha}\alpha r +\dfrac{2^{1+\alpha}r}{\pi},\quad \alpha\ge 2.
\end{cases}
\end{eqnarray*}
 by using Bernoulli inequality
$(1-r)^{1+\alpha}\geq 1-(1+\alpha) r$ for $r\in[0,1]$.
The proof is complete.
\end{pf}

\begin{pf}[Proof of Theorem \ref{thm:swz}]
In view of the form \eqref{eq:fpa} and \eqref{eq:exp2} in Lemma \ref{lem:exp},
we obtian
\begin{equation*}\label{eq:f1}
\begin{aligned}
|f(z)|=|\Pa [f^*](z)|
&=\frac1{2\pi}\left|\int_0^{2\pi} \frac{(1-|z|^2)^{\alpha+1}f^*(e^{i\theta})}{(1-ze^{-i\theta})
(1-\bar{z}e^{i\theta})^{\alpha+1}}\dth\right| \\
&\le||f^*||_{\infty}(1-|z|^2)^{\alpha+1}
\frac{1}{2\pi}\int_0^{2\pi} \frac{d\theta}{|1-ze^{-i\theta}|^{\alpha+2}}\\
&=\Gauss\left(-\frac{\alpha}{2},-\frac{\alpha}{2};1;|z|^{2}\right)||f^*||_{\infty},
\end{aligned}
\end{equation*}
for $z\in\D$.
\end{pf}

\section{Proof of Schwarz-Pick type inequalities}
This section is devoted to the proofs of Theorem \ref{thm:swz-pk}
and Corollary \ref{cor:swz-pk}.

\begin{pf}[Proof of Theorem \ref{thm:swz-pk}]
Denote $r=|z|$, $\xi=ze^{-i\theta}$ and
$$R(\alpha)=\frac{2r(1+\alpha)(1+\alpha r^2)}{(1+\alpha r^2)^2+(1+\alpha)^2r^2}$$
for the sake of simplicity. Note that $R(0)=2r/(1+r^2):=R$.
By observing the form of $P_\alpha(z)$ in (\ref{eq:knl}),
we compute its partial derivatives as follows:
\begin{equation}\label{eq:zbar}
\left|\frac\partial{\partial\bar z} P_\alpha(ze^{-i\theta})\right|
=\frac{(1-|z^2|)^\alpha}{|1-ze^{-i\theta}|^{\alpha+2}}(1+\alpha)
=\frac{(1+\alpha)(1-r^2)^\alpha}{|1-\xi|^{\alpha+2}}
\end{equation}
and
\begin{equation}\label{eq:z}
\begin{aligned}
\left|\frac\partial{\partial z} P_\alpha(ze^{-i\theta})\right|
&=\frac{(1-|z^2|)^{\alpha}}{|1-ze^{-i\theta}|^{\alpha+3}}\left|(1+\alpha)\bar z(z-e^{i\theta})+1-|z|^2\right|\\
&=\frac{(1-r^2)^{\alpha}}{|1-\xi|^{\alpha+3}}\left|(1+\alpha)(r^2-\xi)+1-r^2\right|.
\end{aligned}
\end{equation}

Then a direct application of the equations (\ref{eq:zbar}) and \eqref{eq:exp2} in Lemma \ref{lem:exp}
yields that
\begin{equation}\label{eq:partial}
\begin{aligned}
|f_{\bar z}(z)|&=\frac1{2\pi}\left|\frac{\partial}{\partial\bar z}\int_0^{2\pi}P_\alpha(ze^{-i\theta})f^*(e^{i\theta})\dth\right|\\
&=\frac1{2\pi}\left|\int_0^{2\pi}\frac{\partial}{\partial\bar z}P_\alpha(ze^{-i\theta})f^*(e^{i\theta})\dth\right|\\
&\leq \frac{(1+\alpha)||f^*||_\infty}{2\pi}
\int_0^{2\pi}\frac{(1-r^2)^{\alpha}}
{|1-ze^{-i\theta}|^{\alpha+2}}\dth\\
&=\frac{(1+\alpha)||f^*||_\infty}{1-r^2}
\Gauss\left(-\frac{\alpha}{2},-\frac{\alpha}{2};1;r^{2}\right),
\end{aligned}
\end{equation}
for $z\in\D$.

On the other side, we make the following estimate by using
 the expression \eqref{eq:z}.
\begin{equation}\label{eq:partial1}
\begin{aligned}
|f_z(z)|&=\frac1{2\pi}\left|\frac{\partial}{\partial z}\int_0^{2\pi}P_\alpha(ze^{-i\theta})f^*(e^{i\theta})\dth\right|\\
&\le\frac{||f^*||_\infty}{2\pi}\int_0^{2\pi}\left|\frac{\partial}{\partial z}P_\alpha(ze^{-i\theta})\right|\dth\\
&=(1-r^2)^{\alpha}\frac{||f^*||_\infty}{2\pi}
\int_0^{2\pi}\frac{\left|(1+\alpha)(r^2-\xi)+1-r^2\right|}{|1-\xi|^{\alpha+3}}\dth.
\end{aligned}
\end{equation}

We divide the proof into two cases according to the sign of $\alpha$,
in order to estimate the size of  $|f_z(z)|$.

(I) The case $\alpha\geq 0$.

In combination of the inequality \eqref{eq:partial1}
and the identity \eqref{eq:exp} in Lemma \ref{lem:exp},
we have
 \begin{equation}\label{eq:partialz}
\begin{aligned}
|f_z(z)|&\leq \frac{||f^*||_\infty}{2\pi}\frac{(1-r^2)^{\alpha}\sqrt{(1+\alpha r^2)^2+(1+\alpha)^2r^2}}{(1+r^2)^{\frac{\alpha+3}{2}}}
\int_0^{2\pi}\frac{\left(1-R(\alpha)\cos\theta\right)^{\frac{1}{2}}}
{(1-R\cos\theta)^{\frac{\alpha+3}{2}}}\dth\\
&\le \frac{||f^*||_\infty(1+\alpha)(1-r^2)^{\alpha}}{(1+r^2)^{\frac{\alpha+2}{2}}}
\sum_{n=0}^{\infty}\sum_{m=0}^{2n}\binom{\frac{1}{2}}{m}\binom{-\frac{\alpha+3}{2}}{2n-m}
R(\alpha)^mR^{2n-m}\frac{(2n)!}{4^n(n!)^2}.
\end{aligned}
\end{equation}
since $(1+\alpha r^2)^2+(1+\alpha)^2r^2\leq (1+\alpha)^2(1+r^2)$, as $\alpha\geq0$.
An elementary calculation generates that
$$
R'(\alpha)
=\frac{2r(1-\alpha^2r^2)(1-r^2)^2}{\left[(1+\alpha r^2)^2+(\alpha+1)^2r^2\right]^2},
\quad
R\left(\frac{1}{r}\right)=1\geq R
$$
and $\displaystyle \lim_{\alpha\to+\infty}R(\alpha)=R$
which imply that
$
R(\alpha)\geq \min\{R(0),\displaystyle \lim_{\alpha\to+\infty}R(\alpha)\}=R
$
for $\alpha\geq0$.
Thus we find that
\begin{equation*}\label{eq:positive}
\sum_{m=0}^{2n}\binom{\frac{1}{2}}{m}\binom{-\frac{\alpha+3}{2}}{2n-m}
R(\alpha)^mR^{2n-m}
\leq \sum_{m=0}^{2n}\binom{\frac{1}{2}}{m}\binom{-\frac{\alpha+3}{2}}{2n-m}
R^{2n}
=\frac{\left(1+\frac{\alpha}{2}\right)_{2n}}{(2n)!}R^{2n},
\quad \alpha\geq0,
\end{equation*}
since the coefficients $\binom{\frac{1}{2}}{m}\binom{-\frac{\alpha+3}{2}}{2n-m}$ are negative for $m=1,2,\ldots,2n$.
Applying the last inequality to the estimate \eqref{eq:partialz}, we infer from the transforms (\ref{eq:Euler}) and \eqref{eq:quadratic} that
\begin{eqnarray*}
|f_z(z)|
&\leq&
 \frac{||f^*||_\infty(1+\alpha)(1-r^2)^{\alpha}}{(1+r^2)^{\frac{\alpha+2}{2}}}
\sum_{n=0}^{\infty}\left(1+\frac{\alpha}{2}\right)_{2n}
\frac{R^{2n}}{4^n(n!)^2}\\
&\leq& \frac{||f^*||_\infty(1+\alpha)(1-r^2)^{\alpha}}{(1+r^2)^{1+\frac{\alpha}{2}}}
\Gauss(\frac{\alpha}{4}+\frac{1}2,\frac{\alpha}4+1;1;R^2)\\
&=&\frac{||f^*||_\infty(1+\alpha)(1+r^2)^{\frac{\alpha}{2}}}{1-r^2}
\Gauss\left(\frac{1}2-\frac{\alpha}{4},-\frac{\alpha}4;1;R^2\right)\\
&=&\frac{||f^*||_\infty(1+\alpha)}{1-r^2}
\Gauss\left(-\frac{\alpha}{2},-\frac{\alpha}2;1;r^2\right).
\end{eqnarray*}
Finally we obtain the inequality \eqref{eq:esti1} for $\alpha\geq0$
by using the above estimate and the inequality \eqref{eq:partial}.

(II) The case $-1<\alpha<0$.

Note that the M\"{o}bius transform $\left[(1+\alpha)(r^2-\xi)+1-r^2\right]/(1-\xi)$
of $\xi$ satisfies that
\begin{eqnarray*}
\left|\frac{(1+\alpha)(r^2-\xi)+1-r^2}{1-\xi}\right|
&\leq& \max\left\{\frac{(1+\alpha)(r^2-r)+1-r^2}{1-r},
\frac{(1+\alpha)(r^2+r)+1-r^2}{1+r}\right\}\\
&<&1-\alpha,
\end{eqnarray*}
for $r=|\xi|\in[0,1)$ if $-1<\alpha<0$.
Thus we deduce from \eqref{eq:partial1} and \eqref{eq:exp2} in Lemma \ref{lem:exp} that
\begin{align*}
|f_z(z)|
&\leq||f^*||_\infty\frac{(1-\alpha) (1-r^2)^{\alpha}}{2\pi}
\int_0^{2\pi}\frac{\dth }{|1-\xi|^{\alpha+2}}\\
&=||f^*||_\infty\frac{1-\alpha}{1-r^2}\Gauss\left(-\frac{\alpha}{2},-\frac{\alpha}2;1;r^2\right).
\end{align*}
Therefore we arrive at the claimed assertion
\begin{align*}
&||D_f(z)||
\le\frac{2||f^*||_\infty}{1-r^2}\Gauss\left(-\frac{\alpha}{2},-\frac{\alpha}2;1;r^2\right)
\end{align*}
by \eqref{eq:partial} and the above inequality for $-1<\alpha<0$.

The proof is done.
\end{pf}

\noindent
{\bf Acknowledgements.}
The authors would like to thank
Professor Toshiyuki Sugawa and Dr. Qingtian Shi for discussions.

\def\cprime{$'$} \def\cprime{$'$} \def\cprime{$'$}
\providecommand{\bysame}{\leavevmode\hbox to3em{\hrulefill}\thinspace}
\providecommand{\MR}{\relax\ifhmode\unskip\space\fi MR }
\providecommand{\MRhref}[2]{%
  \href{http://www.ams.org/mathscinet-getitem?mr=#1}{#2}
}
\providecommand{\href}[2]{#2}

\end{document}